\documentclass{amsart}
\usepackage[utf8]{inputenc}
\usepackage{slashed}
\usepackage{textcomp} 
\usepackage{amsmath,amstext,amsopn,amssymb, amsfonts, amsthm, mathtools}  
\usepackage{mathrsfs}
\usepackage{graphicx}
\usepackage{xcolor}
\usepackage{geometry} 
\usepackage{array} 
\usepackage{textcomp} 
\usepackage[all,cmtip]{xy} 
\usepackage{bbm}
\usepackage{varioref}
\usepackage{pdfsync}

\usepackage{enumitem}
\makeatletter
\def\namedlabel#1#2{\begingroup
    #2%
    \def\@currentlabel{#2}%
    \phantomsection\label{#1}\endgroup
}
\makeatother

\usepackage[pdftex,bookmarks,colorlinks,breaklinks]{hyperref}  
\definecolor{dullmagenta}{rgb}{0.4,0,0.4}   
\definecolor{darkblue}{rgb}{0,0,0.4}
\definecolor{darkgreen}{rgb}{0,0.4,0}
\hypersetup{linkcolor=darkblue,citecolor=blue,filecolor=dullmagenta,urlcolor=darkblue} 

\def\XXint#1#2#3{{\setbox0=\hbox{$#1{#2#3}{\int}$}
     \vcenter{\hbox{$#2#3$}}\kern-.5\wd0}}

\newtheorem{theorem}{Theorem}[section]
\newtheorem*{theorem*}{Theorem}
\newtheorem{lemma}[theorem]{Lemma}
\newtheorem*{lemma*}{Lemma}
\newtheorem{proposition}[theorem]{Proposition}
\newtheorem{corollary}[theorem]{Corollary}

\theoremstyle{definition}

\theoremstyle{remark}
\newtheorem{remark}[theorem]{Remark}
\newtheorem{question*}[theorem]{Question}

\numberwithin{equation}{section}

\theoremstyle{theorem}
\newtheorem{ltheorem}{Theorem}

\title[Dyadic coverings with nondoubling measures]{BMO from dyadic BMO for nonhomogeneous measures}

\author[J.M.\ Conde-Alonso]{Jos\'e M. Conde-Alonso}
\address{\noindent Department of Mathematics, Kassar House, 151 Thayer Street \newline \indent Brown University, 02903 Providence RI, USA}
\email{jconde@math.brown.edu}

\subjclass[2010]{Primary: 46E30. Secondary: 42B20, 46E40}
\keywords{BMO, one third trick, nondoubling measures}

\makeindex

\begin{document}

\allowdisplaybreaks

\begin{abstract}
The usual one third trick allows to reduce problems involving general cubes to a countable family. Moreover, this covering lemma uses only dyadic cubes, which allows to use nice martingale properties in harmonic analysis problems. We consider alternatives to this technique in spaces equipped with nonhomogeneous measures. This entails additional difficulties which forces us to consider martingale filtrations that are not regular. The dyadic covering that we find can be used to clarify the relationship between martingale BMO spaces and the most natural BMO space in this setting, which is the space RBMO introduced by Tolsa.
\end{abstract}

\maketitle


\section*{Introduction}

Dyadic coverings are useful tools in harmonic analysis. When a problem involves manipulating a family of general cubes or balls, many times it is useful to restrict one's attention to a discrete |countable| family of them. Among such families, the most useful are those formed by dyadic cubes. To relate general cubes to dyadic ones, one uses a version of the so called one third trick. Roughly, this states that there exists a finite number of dyadic families such that given any cube $Q$, there exists a dyadic cube $R$ in one of them such that $Q \subset R$ and $\ell(Q) \sim \ell(R)$, or equivalently $|Q|\sim |R|$. The idea behind the one third trick goes back at least to the work of Christ, but more modern approaches have improved or used variants of it, such as \cite{mei2003,conde2013,hytonen-lacey-perez}. Dyadic covering lemmata have found many applications to harmonic analysis, among which we will highlight two: first, the relationship between BMO and its dyadic (or martingale) counterpart is an almost immediate corollary (see \cite{mei2003}). Second, the theory of sparse domination (see \cite{lerner2013JAM}) initiated by Lerner and which has grown tremendously in the last few years seems to require some version of the one third trick.   

The goal of this paper is to study dyadic-like covering arguments in a different context, also motivated by harmonic analysis. In particular, we consider them in the context of nonhomogeneous harmonic analysis. We work on $\mathbb{R}^d$ equipped with a measure $\mu$ of $n$-polynomial growth. This means that $\mu$ is a Radon measure that satisfies
$$
\mu(B(x,r)) \leq C_\mu r^n, \mu \; a.e. \; x.
$$
Without loss of generality, we will always assume that $C_\mu=1$. Measures of $n$-polynomial growth appear naturally in the study of analytic capacity or rectifiability, where harmonic analysis tools have proved to be very useful \cite{tolsaPainleve,nazarov-tolsa-volberg2014}. The main difficulty that measures of of $n$-polynomial growth will pose to us is that they need not be \textit{doubling}, which means that the measure of a ball and the measure of a fixed dilate of it need not be comparable. This means that the usual one third trick is not going to be useful for us, since $\ell(Q) \sim \ell(R)$ may no longer imply $\mu(Q) \sim \mu(R)$. The best we can hope for is to discretize the family of doubling cubes, which is often enough in the applications. Given constants $\alpha>1 $ and $\beta> 0$ we say that a cube or ball $Q$ is $(\alpha,\beta)$-doubling if $\mu(\alpha Q) \leq \beta\mu(Q)$. We can certainly cover $(\alpha,\beta)$-doubling cubes by dyadic cubes while keeping the key property $\mu(Q)\sim \mu(R)$, as we show in appendix \ref{secappendix}. However, this has limited applications because the resulting dyadic families are not complete and do not form a filtration of $\mathbb{R}^d$. Therefore, we resort to more complicated sets for which we can keep the martingale properties.

\begin{ltheorem}\label{thmA}
Fix $\alpha>60$, $C_0> (6 \sqrt{d} \alpha)^d$ and set $\alpha_0 = 6 \sqrt{d} \alpha$. There exist $N=N(d)$ atomic filtrations $\Sigma^1, \Sigma^2, \cdots, \Sigma^N$ of $\mathrm{supp}(\mu) \subset \mathbb{R}^d$ with the following properties:
\begin{enumerate}
\item For each atom $T \in \Sigma^k$, $1\leq k \leq N$, there exists a ball $B_T$ such that 
$$
B_T \cap \mathrm{supp}(\mu) \subset T \subset 30B_T \cap \mathrm{supp}(\mu) \;\; \mathrm{and} \;\; \mu(\alpha B_T) \leq C_0\mu(B_T). 
$$
\item For each $(\alpha_0,C_0)$-doubling cube $Q$ in $\mathbb{R}^d$, there exists $1\leq k \leq N$ and $T\in \Sigma^k$ such that 
$$
Q \cap \mathrm{supp}(\mu) \subset T, \;\; \ell(Q) \sim r(B_T) \;\; \mathrm{and} \;\; \mu(Q) \sim \mu(T).
$$
\end{enumerate}
\end{ltheorem}

The number $N=N(d)$ in the statement does not depend on the value of $\alpha_0$, but the filtrations themselves do. As the statement suggests, the sets $T \in \Sigma$ will play the role of dyadic cubes of the same dimension as the measure (recall that they are similar to balls). The filtrations $\Sigma^k$ satisfy additional properties, see section \ref{sec1} for details and for precise definitions, that we postpone for now. The structure of each $\Sigma^k$ is well adapted to the study of problems related to Calder\'on-Zygmund operators (see \cite{condealonso-parcet2018}), properties of the harmonic measure (see \cite{azzam-hofmann-martell-mayboroda-mourgoglou-tolsa-volberg}), or rectifiability of sets (see \cite{condealonso-mourgoglou-tolsa2018,david-mattila2000}), to name a few potential scenarios where our result may be of use. To prove theorem \ref{thmA} we will elaborate on the construction of \cite{condealonso-parcet2018}, where a version of the celebrated lattice by David and Mattila \cite{david-mattila2000} is tweaked to get further properties. The main challenge here is to be able to construct several dyadic-like lattices at the same time so that all doubling balls in the space are well adapted to them, something we believe has not been considered before.

As we said above, coverings by cubes in finitely many dyadic filtrations have found many different applications and are nowadays part of the standard set of tools in harmonic analysis. As an application of our methods, we shall prove versions of \textit{BMO from dyadic BMO} results (see \cite{garnett-jones}) for nondoubling measures. In the nonhomogeneous setting, there are natural candidates for both spaces. The role of BMO is usually played by the space RBMO that was introduced by Tolsa in \cite{tolsaRBMO}. This space, whose definition is postponed to section \ref{sec2}, satisfies two key features of the classical BMO class: interpolation with the $L^p$ scale and boundedness of Calder\'on-Zygmund operators. Its dyadic counterpart is denoted by $\mathrm{RBMO}_\Sigma$ and was introduced in \cite{condealonso-parcet2018}. There it was shown that this space, which is nothing but the martingale $\mathrm{BMO}$ space associated with a David-Mattila filtration $\Sigma$, satisfies
$$
\mathrm{RBMO} \subsetneq \mathrm{RBMO}_\Sigma.
$$
Ideally, one would like to prove a statement similar to 
$$
\mathrm{RBMO} = \bigcap_{j=1}^N \mathrm{RBMO}_{\Sigma^j},
$$
with equivalent norms. This would be a direct generalization of the main result in \cite{mei2003}. However, this seems to be false for reasons that will become clear later in the text. What we will do is to define slight variations $\mathrm{RBMO}_\Sigma^*$ of $\mathrm{RBMO}_\Sigma$ (whose precise definition we postpone again to section \ref{sec2}) for which the result holds. This is our second result:

\begin{ltheorem}\label{thmBMOdyad}
Let $\{\Sigma^j\}_{1\leq j \leq N}$ be the family of filtrations given by theorem \ref{thmA}. We have
$$
\mathrm{RBMO} = \bigcap_{j=1}^N \mathrm{RBMO}_{\Sigma^j}^*,
$$
with equivalent norms. Moreover, for all filtrations $\Sigma$ we have that $\mathrm{RBMO}_\Sigma^* \subset \mathrm{RBMO}_\Sigma$.
\end{ltheorem}

The second part of theorem \ref{thmBMOdyad} implies that the variations $\mathrm{RBMO}_\Sigma^*$ still interpolate with the $L^p$ scale while the first implies that $L^2$ bounded Calder\'on-Zygmund operators map $L^\infty(\mu)$ into $\mathrm{RBMO}_\Sigma^*$. This means that our new spaces $\mathrm{RBMO}_\Sigma^*$ are suitable counterparts of dyadic BMO in the nonhomogeneous setting. Finally, we remark that in the same spirit as in \cite{condealonso-parcet2018}, all our results generalize directly to the operator valued setting with very minor changes that we omit and that can be easily figured out by the interested reader. 

\subsection*{Remark about notation} In this paper we use different kinds of martingale filtrations. To avoid confusion, we shall use the letter $\mathcal{D}$ (maybe with superscripts) to denote usual dyadic filtrations. We will use $\mathscr{D}$ for David-Mattila filtrations while we keep $\Sigma$ for David-Mattila filtrations where all the David-Mattila cubes are doubling (see section \ref{sec1} for details).

\subsection*{Acknowledgment} The author wants to thank Jill Pipher for initially pointing out the main question addressed in this note and Alex Barron for finding a key mistake in an earlier version of this manuscript.


\section{The one third trick for nondoubling measures} \label{sec1}

\subsection{David-Mattila filtrations} The sets that we are going to deal with are the ones that appear in the statement of theorem \ref{thmA}. We will call them David-Mattila cubes in what follows, and they are sets $Q$ which have the following property: there exists a ball $B_Q$ such that $B_Q \cap \mathrm{supp}(\mu) \subset Q \subset 30B_Q \cap \mathrm{supp}(\mu)$. Given $\alpha>30$, we say that a David-Mattila cube $Q$ is $(\alpha,\beta)$-doubling if its associated ball $B_Q$ is $(\alpha,\beta)$-doubling, that is, if $\mu(\alpha B_Q) \leq \beta \mu(B_Q)$. 

\begin{proposition}[Theorem 3.2 in \cite{david-mattila2000}] \label{propDM}
Let $\mu$ be a Radon measure on $\mathbb{R}^{d}$ of $n$-polynomial growth. Fix $\alpha>1$. Then, there exist $C_0=C_0(\alpha)>1$ and $A_0>5000\,C_0$, such that for each choice of $1\leq \tilde{C}_0 < _0$ there exists a sequence $\mathscr{D}= \cup_{k} \mathscr{D}_k$ of partitions of $\mathrm{supp}(\mu)$ into
Borel subsets $Q$ with the following properties:
\begin{itemize}

\item If $k<\ell$, $Q\in\mathscr{D}_{k}$, and $R\in\mathscr{D}_{\ell}$, then either $Q\cap R=\varnothing$ or else $Q\subset R$.

\item For each $k$ and each cube $Q\in\mathscr{D}_{k}$, there is a ball $B_Q=B(x_Q,r(Q))$ such that
$$x_Q\in \mathrm{supp}(\mu), \qquad C_0^{-1} \tilde{C}_0 A_0^{-k}\leq r(Q)\leq \tilde{C}_0 \,A_0^{-k},$$
$$\mathrm{supp}(\mu)\cap B(Q)\subset Q\subset \mathrm{supp}(\mu)\cap 30B(Q)=\mathrm{supp}(\mu) \cap B(x_Q,30r(Q)),$$
and the balls $5B(Q)$, $Q\in\mathscr{D}_{k}$, are disjoint.
\item The balls $\frac12B_Q$ and $\frac12B_{Q'}$ associated with $Q \not= Q'$ are disjoint unless $Q \subset Q'$ or $Q' \subset Q$.
\item The cubes $Q\in\mathscr{D}_{k}$ have small boundaries: for each $Q\in\mathscr{D}_{k}$ and each
integer $\ell\geq0$, set
$$N_{\ell}^{ext}(Q)= \{x\in \mathrm{supp}(\mu)\setminus Q:\,\mathrm{dist}(x,Q)< A_0^{-k-\ell}\},$$
$$ N_{\ell}^{int}(Q)= \{x\in Q:\,\mathrm{dist}(x,\mathrm{supp}(\mu)\setminus Q)< A_0^{-k-\ell}\},$$
and
$$N_{\ell}(Q)= N_{\ell}^{ext}(Q) \cup N_{\ell}^{int}(Q).$$
Then
\begin{equation*}
\mu(N_{\ell}(Q))\leq (C^{-1}C_0^{-3d-1}A_0)^{-\ell}\,\mu(90B_Q).
\end{equation*}

\item If $Q\in\mathscr{D}_{k}$ is not $(\alpha,C_0)$-doubling then $r(Q)=A_0^{-k}$ and
\begin{equation*}
\mu(\alpha B_Q)\leq C_0^{-\ell}\,\mu(\alpha^{\ell+1}B_Q)\quad
\mbox{for all $\ell\geq1$ with $\alpha^\ell\leq C_0$.}
\end{equation*}
\end{itemize}

\end{proposition}

The roles of the various constants above are the following: $A_0$ is the quotient between the respective side lengths of a David-Mattila cube and its children. $\alpha$ and $C_0$ are the relevant doublingness constants and $\tilde{C}_0$ is just a parameter that one can vary to make sure that any given side length can be realized in a David-Mattila filtration. The last item in proposition \ref{propDM} implies that if $Q \in \mathscr{D}$ is $(\alpha,C_0)$-doubling and $\widehat{Q}$ is the smallest $(\alpha,C_0)$-doubling cube in $\mathscr{D}$ that contains it properly, then
\begin{equation}\label{eqweakregularity}
\int_{\alpha B_{\widehat{Q}} \setminus \alpha B_Q} \frac{1}{|x_Q-y|^n} d\mu(y) \lesssim_{\alpha,A_0} 1.
\end{equation}
This property is useful in applications to harmonic analysis, as we will see in section \ref{sec2}. In \cite{condealonso-parcet2018} the David-Mattila construction is modified to yield a filtration with a few additional properties. In particular, the side length of the cubes is no longer uniformly bounded from above and all the David-Mattila cubes in the filtration are doubling. In addition, the construction of \cite{condealonso-parcet2018} allows one to choose a particular $(\alpha,\beta)$-doubling ball to be a ball $B_Q$ associated to some David-Mattila cube. The result is the following:

\begin{proposition}[Theorem A in \cite{condealonso-parcet2018}] \label{propDMCAP}
Let $\mu$ be a measure of $n$-polynomial growth on $\mathbb{R}^d$. Fix $\alpha>1$ and an $(\alpha,C_0)$-doubling ball $B$.Then there exists a positive constant $C_0=C_0(\alpha)$  and a two-sided filtration of atomic $\sigma$-algebras of $\mathrm{supp}(\mu)$ generated by a sequence of nested partitions $\Sigma=\{\Sigma_k: k \in \mathbb{Z}\}$ that satisfies the following properties:
\begin{itemize}
\item The union of $L^\infty(\mathbb{R}^d, \sigma(\Sigma_k), \mu)$ is weak-$*$ dense in $L^\infty(\mu)$. That is, we have
$$
\lim_{Q \to x, Q \in \Sigma} \frac{1}{\mu(Q)} \int_Q f \; d\mu = f(x), \; \mu-\mbox{a.e.} \; x.
$$

\item If $Q \in \Sigma$, then $Q$ is a David-Mattila cube, so there exists an $(\alpha,C_0)$-doubling ball $B_Q$ with $B_Q \subset Q \subset 30B_Q$.

\item There exists $Q \in \Sigma$ such that $B_Q=B$.
\end{itemize}
\end{proposition}
   
\begin{remark}
Proposition \ref{propDMCAP} is proved using a similar construction than proposition \ref{propDM}. As a result, the thin boundaries of the David-Mattila cubes are preserved, and so is \eqref{eqweakregularity}, which justifies our notation $\widehat{Q}$ for the smallest doubling ancestor of $Q$ (it is strictly the father in the filtration $\Sigma$).
\end{remark}

\begin{remark}\label{schemeOfProof} To prove proposition \ref{propDMCAP}, one needs to slightly modify the proof of proposition \ref{propDM}. The main change that one needs to implement affects only the first step of the proof, which is where the balls $B_Q$ associated to the David-Mattila cubes are chosen. Unfortunately, we cannot use the statement of proposition \ref{propDMCAP} to prove theorem \ref{thmA}. Instead, we need to modify again the first step of the proof of proposition \ref{propDM} to carefully choose the balls $B_Q$ in each of the filtrations that we construct. We shall justify that this is possible below. After that, we can just carry out the rest of the argument in section $1$ of \cite{condealonso-parcet2018} step by step. 
\end{remark}

To select the balls $B_Q$ associated to David-Mattila cubes in the filtrations $\Sigma^k$ that we will construct, we shall use the following easy modification of the $5R$ covering lemma:

\begin{lemma}\label{lemmaVitali}
Let $E \subset \mathbb{R}^d$ be a set and $\mathcal{B}_0$ a countable family of balls of the same radius $R$ and which are pairwise disjoint. Assume that for each $x\in E \setminus \cup_{B \in \mathcal{B}_0} B$ there exists a ball $B_x$ with radius $r_x\leq R$. Then, there exists a countable subcollection $\mathcal{B}_1 \subset \{B_x\}_{x\in (E \setminus \cup_{B \in \mathcal{B}_0} B)}$ satisfying:
\begin{itemize}
\item The balls in $\mathcal{B}:= \mathcal{B}_0 \cup \mathcal{B}_1$ are pairwise disjoint.
\item $E \subset \cup_{B \in \mathcal{B}} 5B$.
\end{itemize} 
\end{lemma}

\begin{proof}
The proof is just the usual one of the classical $5R$ covering theorem applied to the family of balls 
$$
\mathcal{B}_0 \cup \{B_x\}_{x\in (E \setminus \cup_{B \in \mathcal{B}_0} B)},
$$
with the only modification that we always pick the balls in $\mathcal{B}_0$ first. This can be done because their radii are maximal and they are disjoint.
\end{proof}

\subsection{Proof of theorem \ref{thmA}} We first prove item (1). This is however immediate since we will use $N$  filtrations as the one in proposition \ref{propDMCAP}. We will specify how we choose the $N$ families and justify the value of $\alpha$ below.

Next, in order to prove (2) we first discretize the family of cubes that we deal with. This is done via the usual one third trick. We use the version in appendix \ref{secappendix}: given an $(\alpha_0,C_0)$-doubling cube $Q$ we can use lemma \ref{lemma2cubos} to find a dyadic cube $Q'$ belonging to one of $3^d$ dyadic filtrations $\mathcal{D}^1, \cdots, \mathcal{D}^{3^d}$ such that $Q \subset Q' \subset 6Q$. These inclusions imply 
\begin{equation}\label{computationDoubling}
\mu\left(\frac{\alpha_0}{6}Q'\right) \leq \mu(\alpha_0 Q) \leq C_0 \mu(Q) \leq C_0 \mu(Q'),
\end{equation}
so $Q'$ is $(\alpha_0',C_0)$-doubling, with $\alpha_0'=\alpha_0/6$. Therefore, we have reduced (2) to proving the same statement for $(\alpha_0',C_0)$-doubling cubes belonging to the union of $3^d$ dyadic filtrations. Note that this is a countable family.

We now turn to the choice of the $N$ filtrations. We first show an easier case: if we admit $N=\infty$ then this is directly given by proposition \ref{propDMCAP}. Indeed, for each $(\alpha_0',C_0)$-doubling cube $Q$ belonging to one of the $3^d$ dyadic filtrations $\mathcal{D}^1, \cdots, \mathcal{D}^{3^d}$, consider the ball $B(Q)$ with the same center as $Q$ and radius $r= \sqrt{d} \; \ell(Q)/2$. By definition, we have $Q \subset B(Q) \subset \sqrt{d} Q$ and so by a computation similar to \eqref{computationDoubling} $B(Q)$ is $(\alpha_0'',C_0)$-doubling with $\alpha_0''= \alpha_0'/\sqrt{d}$. Therefore, we may apply proposition \ref{propDMCAP} with $\alpha=\alpha_0''$ and $B=B(Q)$ to get a filtration that we denote $\Sigma^Q$. Then we just consider all the filtrations $(\Sigma^Q)_{Q}$ and we check both items in the statement of theorem \ref{thmA}: on the one hand, as we said above, (1) follows because our filtrations $\Sigma^{Q'}$ have $(\alpha,C_0)$-doubling associated balls, with $\alpha = \alpha_0''= \alpha_0/(6\sqrt{d})$. One the other hand, for each $(\alpha_0,C_0)$-doubling cube $Q$, we can pick a David-Mattila cube $T$ that satisfies all the properties in (2) as follows: $T$ is the David-Mattila cube of the filtration $\Sigma^{Q'}$ associated with the ball $B(Q')$, where $Q'$ is the $(\alpha_0',C_0)$-doubling dyadic cube given by the application of lemma \ref{lemma2cubos} to $Q$.
\

\

\
We now push the argument forward a little bit to show that we may actually consider only a finite number of filtrations. To that end, as we explained in remark \ref{schemeOfProof}, we need to follow the scheme of the proof of proposition \ref{propDMCAP} with a new selection procedure to choose the balls $B_T$ associated to David-Mattila cubes $T$. We start again discretizing the family of cubes and we assume that we have fixed one of the $3^d$ dyadic families that we denote $\mathcal{D}$. Recall that we only need to worry about cubes in $\mathcal{D}$ that are $(\alpha_0',C_0)$-doubling. Our next step is to partition the $(\alpha_0',C_0)$-doubling cubes in $\mathcal{D}$ into finitely many families $\{\mathcal{F}_j\}_{1 \leq j \leq N_0}$ according to the following two rules:

\begin{enumerate}
\item[(i)] If two different cubes $Q$ and $Q'$ with the same side length $2^{-k}$ belong to $\mathcal{F}_j$, the distance between them is at least $5\cdot2^{-k}\sqrt{d}$. 
\item[(ii)] If two different cubes $Q$ and $Q'$ with different respective side lengths $2^{-k}$ and $2^{-\ell}$, $k<\ell$, belong to $\mathcal{F}_j$, then $2^{-k+\ell} =A_0^m$ for some positive integer $m$. 
\end{enumerate}
The above splitting is obviously possible, and the smallest number $N_0$ depends at most exponentially on the dimension $d$. Recall that $A_0$ is precisely the distance between generations in proposition \ref{propDM}\footnote{Without loss of generality, we may (and do) assume that the constants $C_0$ and $A_0$ are of the form $C_0=2^a$ and $A_0=2^b$ for positive integers $a$ and $b$}. Using that, for each $j$ we now construct a filtration $\Sigma^{j,\mathcal{D}}$ such that (2) holds for all cubes $Q \in \mathcal{F}_j$ and some David-Mattila cubes in $\Sigma^{j,\mathcal{D}}$. This will end the proof by applying the same procedure to each of the dyadic systems $\mathcal{D}$ to end up with $N= 3^d N_0$ different David-Mattila filtrations.

Fix now the index $j$. As before, for each $Q \in \mathcal{F}_j$, denote by $B(Q)$ the ball with the same center as $Q$ and radius $r= \sqrt{d}\; \ell(Q)/2$. Note that by rule (i) for cubes in $\mathcal{F}_j$, $5B(Q) \cap 5B(Q')=\emptyset$ if $Q$ and $Q'$ belong to the same dyadic generation. Consider now each dyadic generation $\mathscr{D}_k$ such that $\mathcal{F}_j \cap \mathscr{D}_k\not=\emptyset$. We take $\mathcal{B}_0=\mathcal{B}_0(k)= \{5B(Q)\}_{Q \in \mathcal{F}_j \cap \mathscr{D}_k}$ and for each $x \in \mathrm{supp}(\mu) \setminus \cup_{B\in \mathcal{B}_0} B$ we take $B_x$ to be the $5$-fold dilate of the largest $(\alpha_0'',C_0)$-doubling ball centered at $x$ of radius $r(x)$ such that 
$$ 
C_0^{-1} 2^{-k-1}\sqrt{d} \leq r(x) \leq 2^{-k-1}\sqrt{d}.
$$
If there is no such doubling ball then take $r(x) = C_0^{-1} 2^{-k-1}\sqrt{d}$. Then apply lemma \ref{lemmaVitali} to $E=\mathrm{supp}(\mu)$. After applying lemma \ref{lemmaVitali} we have a collection $\mathcal{B}_1=\mathcal{B}_1(k)$ of balls that contains $\mathcal{B}_0$. The radii $r$ of the balls satisfy
$$
C_0^{-1} A_0^{m} \tilde{C}_0 \leq r \leq  A_0^{m} \tilde{C}_0,
$$
for some $1 \leq \tilde{C}_0 \leq C_0$ that depends on $j$ but not on $k$; indeed, it depends on $k\; \mathrm{mod} \; A_0$ which is constant within $\mathcal{F}_j$ because of rule (ii). Therefore, each family of balls $\mathcal{B}_1(k)$ is a suitable family of balls $B_T$ associated to David-Mattila cubes of one generation. From this point, we can follow step by step the argument in \cite{condealonso-parcet2018} to construct a filtration $\Sigma^{j,\mathcal{D}}$ with the properties of proposition \ref{propDMCAP}. We repeat the construction for each $j$ and then for $\mathcal{D}= \mathcal{D}^m$, $1 \leq m \leq N_0$ to get $N$ David-Mattila filtrations $\Sigma^{j, \mathcal{D}^m}$, $1\leq j \leq N_0$, $1 \leq m 
\leq 3^d$.

We can finally check both items in the statement of theorem \ref{thmA}: (1) follows by construction and for each $(\alpha_0,C_0)$-doubling cube $Q$, we can pick a David-Mattila cube $T \in \cup_{j,m} \Sigma^{j, \mathcal{D}^m}$ that satisfies all the properties in (2) in a similar way as in the case $N=\infty$. We first choose $m_0$ so that there is an $(\alpha_0',C_0)$-doubling dyadic cube $Q' \in \mathcal{D}^{m_0}$ with $Q \subset Q' \subset 6Q$. Second, we choose $j_0$ so that $Q' \in \mathcal{F}_{j_0}$. Then, $T$ is the David-Mattila cube of the filtration $\Sigma^{j_0, \mathcal{D}^{m_0}}$ such that $B_T=B(Q')$. It is immediate to check that all properties in (2) are satisfied and therefore the proof is complete.

\qed


\section{$\mathrm{RBMO}$ from martingale $\mathrm{RBMO}$} \label{sec2}


\subsection{The $\mathrm{RBMO}$ space of Tolsa} In order to define the appropriate $\mathrm{BMO}$ space for measures of $n$-polynomial growth we need to recall a way of comparing two cubes independent of their respective side lengths. In particular, given a pair of cubes or balls $Q$ and $R$, we define
$$
\delta(Q,R)= 1+ \int_{2R\setminus 2Q} \frac{d\mu(y)}{|x_Q -y|^n}.
$$
$\delta(Q,R)$ is a notion of distance between two cubes or balls $Q$ and $R$ with nontrivial intersection. We will always be considering cubes or balls with nontrivial intersection, and so we will not worry about $\delta(Q,R)$ when they are far away. The following easy properties of $\delta$ are going to be useful in the sequel:

\begin{lemma}[see \cite{tolsaRBMO}] \label{lemmaKQR}
The following hold:
\begin{itemize}
\item If $Q \subset R \subset T$, then $\max\{\delta(Q,R), \delta(R,T)\} \leq \delta(Q,T)$.
\item If $\ell(Q) \sim \ell(R)$ then $\delta(Q,R) \sim 1$.
\end{itemize}
\end{lemma}

Fix two constants $\alpha >1$, $\beta > \alpha^d$. We say that a function $f$ belongs to $\mathrm{RBMO}$ if the following quantity is finite:
$$
\|f\|_{\mathrm{RBMO}(\alpha,\beta)} = \max \left\{ \|f\|_{\mathrm{DBMO}(\alpha,\beta)}, \|f\|_{\mathrm{RBMO}_d(\alpha,\beta)} \right\},
$$
where
$$
\|f\|_{\mathrm{DBMO}(\alpha,\beta)} = \sup_{Q \; (\alpha,\beta)-\mathrm{doubling}} \frac{1}{\mu(Q)} \int_Q \left|f -\langle f \rangle_Q \right| \; d\mu
$$
and
$$
\|f\|_{\mathrm{RBMO}_d(\alpha,\beta)} = \sup_{\begin{subarray}{c} Q \subset R \\ Q,R \; (\alpha,\beta)-\mathrm{doubling} \end{subarray}} \frac{\left|\langle f \rangle_Q - \langle f \rangle_R\right|}{\delta(Q,R)}.
$$
For us, $\langle f \rangle_Q$ denotes the integral average with respect to the measure $\mu$:
$$
\langle f \rangle_Q := \frac{1}{\mu(Q)} \int_Q f(x) \; d\mu(x).
$$
It can be seen that the above definition does not really depend on the constants $\alpha$ and $\beta$:

\begin{lemma} [see \cite{tolsaRBMO}] \label{lemmaalphabeta}
Let $\alpha, \alpha_0, \beta, \beta'$ be such that $\alpha,\; \alpha_0 > 1$, $\beta > \alpha^d$ and $\beta' > \alpha_0$. Then 
$$
\|f\|_{\mathrm{RBMO}(\alpha,\beta)} \sim_{\alpha,\alpha_0,\beta,\beta'} \|f\|_{\mathrm{RBMO}(\alpha_0,\beta')}.
$$
\end{lemma}
Because of lemma \ref{lemmaalphabeta}, in what follows we will simply use the terms $\|f\|_{\mathrm{RBMO}}$, $\|f\|_{\mathrm{DBMO}}$ and $\|f\|_{\mathrm{RBMO}_d}$ without explicitly mentioning the associated constants $\alpha$ and $\beta$.


\subsection{Dyadic $\mathrm{RBMO}$ spaces} In \cite{condealonso-parcet2018} a martingale version of $\mathrm{RBMO}$ was introduced: given an appropriately constructed David-Mattila filtration $\mathscr{D}$, the family of elements in $\mathscr{D}$ which are doubling form a new two sided filtration that we denote $\Sigma=\{\Sigma_k\}_{k\in\mathbb{Z}}$. Then we define $\mathrm{RBMO}_\Sigma$ to be the martingale $\mathrm{BMO}$ space associated to the filtration $\Sigma$. The norm in $\mathrm{RBMO}_\Sigma$ is therefore given by
$$
\|f\|_{\mathrm{RBMO}_\Sigma} =\sup_{k} \left\|\mathsf{E}_{\Sigma_k}\left|f-\mathsf{E}_{\Sigma_{k-1}}f\right| \right\|_\infty \sim \sup_{Q \in \Sigma} \left[\frac{1}{\mu(Q)} \int_Q |f -\langle f \rangle_Q| d\mu + \left|\langle f \rangle_Q -\langle f \rangle_{\widehat{Q}}\right|\right].
$$
Above, $\mathsf{E}_{\Sigma_k}$ denotes the conditional expectation with respect to the $\sigma$-algebra generated by $\Sigma_k$ and $\widehat {Q}$ is the father of $Q\in \Sigma_k$ in $\Sigma$, that is, the only atom $R\in \Sigma_{k-1}$ that properly contains $Q$. As we stated in the introduction, $\mathrm{RBMO}_\Sigma$ enjoys two remarkable properties that make it useful (see \cite{condealonso-parcet2018}):

\begin{itemize}
\item It is a martingale $\mathrm{BMO}$ space, so it interpolates with the $L^p$ scale, and its predual is known. 
\item $\mathrm{RBMO} \subset \mathrm{RBMO}_\Sigma$, so operators that are bounded from $L^\infty(\mu)$ to $\mathrm{RBMO}$ are bounded from $L^\infty(\mu)$ to $\mathrm{RBMO}_\Sigma$ as well.
\end{itemize}
We now introduce the modified $\mathrm{RBMO}_\Sigma$ spaces that we will use in the proof of theorem \ref{thmBMOdyad}. It turns out that the quantity $\delta(Q,R)$ captures more information than just the distance in dyadic generations |David-Mattila ones, of course| between $Q$ and $R$. Therefore, we have to take them into account to define the right dyadic $\mathrm{BMO}$ spaces. We start by redefining the quantity $\delta$ as follows: given David-Mattila cubes $Q \subset R$ we set
$$
\delta(Q,R):= 1+ \int_{\alpha B_R\setminus \alpha B_Q} \frac{d\mu(y)}{|x_{B_Q} -y|^n}.
$$
The abuse of notation above is justified by the fact that if $Q$ and $R$ belong to $\Sigma$ then $\delta(Q,R) \sim \delta(B_Q,B_R)$, which follows directly from the definition and lemma \ref{lemmaKQR}. Then, given a doubling David-Mattila filtration $\Sigma$ we set
$$
\|f\|_{\mathrm{RBMO}_{\Sigma}^*} := \sup_{Q \in \Sigma} \left[\frac{1}{\mu(Q)} \int_Q |f -\langle f \rangle_Q| d\mu\right] + \sup_{Q\in \Sigma, j >0}\left|\frac{\langle f \rangle_Q -\langle f \rangle_{Q^{(j)}}}{\delta(Q,Q^{(j)})}\right|.
$$
In the above formula, $Q^{(j)}$ denotes the $j$-th dyadic ancestor of $Q\in \Sigma_k$, that is, the only $R\in \Sigma_{k-j}$ such that $Q \subset R$. From the definition, it immediately follows that $\mathrm{RBMO}_{\Sigma}^* \subset \mathrm{RBMO}_{\Sigma}$.

We take now $\{\Sigma^j\}_{j=1}^N$ to be the family of filtrations given by theorem \ref{thmA}, with $\alpha=480\sqrt{d}$ |the reason of this choice will be clear immediately. We have now defined all the elements in the statement of theorem \ref{thmBMOdyad}, that we restate here:
\begin{theorem}\label{thm23} 
Under the choice of the value of $\alpha$ above, we have
$$
\mathrm{RBMO} = \bigcap_{j=1}^N \mathrm{RBMO}_{\Sigma_j}^*,
$$
with equivalent norms.
\end{theorem}
The rest of this section is entirely devoted to its proof. 


\subsection{Proof of theorem \ref{thm23}} We start with the easier inequality, that is, 
\begin{equation}\label{easyineq}
\|f\|_{\mathrm{RBMO}_{\Sigma^j}^*} \lesssim \|f\|_{\mathrm{RBMO}},
\end{equation}
for any $j$ and any $f\in \mathrm{RBMO}$. Fix $Q \in \Sigma^j$. Denote by $T$ a Euclidean cube such that $30B_Q \subset T$ and $2T\subset \frac{\alpha}{4} B_Q$. Also, let $\widehat{T}$ be a cube such that $60B_{\widehat Q} \subset \widehat{T}$ and $2\widehat{T}\subset \frac{\alpha}{2} B_{\widehat{Q}}$. Then by our choice of $\alpha$ both $T$ and $\widehat{T}$ are $(2, C_0)$-doubling. Also, since $Q \subset \widehat{Q}$ then $T \subset \widehat{T}$. According to lemma \ref{lemmaalphabeta}, we may choose the doublingness constants in the $\mathrm{RBMO}$ norm to be $2$ and $C_0$. On the other hand, by lemma \ref{lemmaKQR} we see that $\delta(Q, \widehat{Q}) \sim \delta(T, \widehat{T})$. By the preceding discussion, we get
\begin{align*}
\frac{1}{\mu(Q)} \int_Q |f - \langle f \rangle_{Q}| d\mu & \leq \frac{1}{\mu(Q)} \int_Q |f - \langle f \rangle_{T}| d\mu + |\langle f \rangle_{T} - \langle f \rangle_{Q}| \\
& \leq 2\frac{1}{\mu(Q)} \int_Q |f - \langle f \rangle_{T}| d\mu \leq 2\frac{\mu(T)}{\mu(Q)} \frac{1}{\mu(T)} \int_T |f - \langle f \rangle_{T}| d\mu \\
& \lesssim \frac{1}{\mu(T)} \int_T |f - \langle f \rangle_{T}| d\mu \leq \|f\|_{\mathrm{RBMO}}. \\
\end{align*}
On the other hand, using the above computation we can also estimate
\begin{align*}
\left|\langle f \rangle_Q -  \langle f \rangle_{\widehat{Q}} \right| & \leq \left|\langle f \rangle_Q -  \langle f \rangle_{T}\right| + \left|\langle f \rangle_T -  \langle f \rangle_{\widehat{T}}\right| + \left|\langle f \rangle_{\widehat{T}} -  \langle f \rangle_{\widehat{Q}}\right| \\
 & \leq \frac{1}{\mu(Q)} \int_Q |f - \langle f \rangle_{T}| d\mu + \left|\langle f \rangle_T -  \langle f \rangle_{\widehat{T}}\right| + \frac{1}{\mu(\widehat{Q})} \int_{\widehat{Q}} |f - \langle f \rangle_{\widehat{T}}| d\mu \\
 & \leq 2\|f\|_{\mathrm{RBMO}} + \delta(T,\widehat{T}) \|f\|_{\mathrm{RBMO}} \lesssim \delta(Q,\widehat{Q}) \|f\|_{\mathrm{RBMO}}. \\
\end{align*}
Taking a supremum over $Q \in \Sigma^j$  and over $1\leq j \leq N$ yields \eqref{easyineq}. We are therefore left with the more difficult inequality, that is,
\begin{equation}\label{difficultineq}
\|f\|_{\mathrm{RBMO}} \lesssim \max_{1 \leq j \leq N} \|f\|_{\mathrm{RBMO}_{\Sigma^j}^*}.
\end{equation}
We now use lemma \ref{lemmaalphabeta} so that we may assume that the cubes that appear in the expression of the $\mathrm{RBMO}$ norm are $(\alpha_0,C_0)$-doubling, where $\alpha_0$ is such that theorem \ref{thmA} holds with David-Mattila cubes that are $(\alpha,C_0)$-doubling. That way, we can make sure that we can cover them by cubes in our filtrations $\Sigma^j$. We estimate the terms in the norm in turn. Fix an $(\alpha_0,C_0)$-doubling cube $T$, and let $Q\in \Sigma^j$ be the David-Mattila cube associated to $T$ via theorem \ref{thmA}, and recall that $T \subset Q$ and $\mu(T) \sim \mu(Q)$. Then we have, as above,
\begin{align*}
\frac{1}{\mu(T)} \int_T |f - & \langle f \rangle_{T} | d\mu \leq \frac{1}{\mu(T)} \int_T |f - \langle f \rangle_{Q}| d\mu + |\langle f \rangle_{Q} - \langle f \rangle_{T}| \\
& \leq 2\frac{1}{\mu(T)} \int_T |f - \langle f \rangle_{Q}| d\mu \lesssim \frac{1}{\mu(Q)} \int_Q |f - \langle f \rangle_{Q}| d\mu \leq \|f\|_{\mathrm{RBMO}}. \\
\end{align*}
We now have to estimate the term $|\langle f \rangle_T - \langle f \rangle_S| \delta(T,S)^{-1}$ for $T\subset S$. In this case, we take $Q \in \Sigma^j$ as the one given by theorem \ref{thmA} such that $S \subset Q$ and $\mu(S) \sim \mu(Q)$. We obviously have that $T \subset Q$, so we may consider the family $\{R\}_{R \in \mathcal{R}}$ of maximal descendants of $Q$ in $\Sigma^j$ that cover $T$ and such that $R \subset 10T$ for all $R \in \mathcal{R}$. Our splitting is now
\begin{align*}
\left|\langle f \rangle_T - \langle f \rangle_S\right| & \leq \left|\langle f \rangle_T - \sum_{R \in \mathcal{R}}\frac{\mu(T\cap R)}{\mu(T)} \langle f \rangle_R\right| + \left|\sum_{R \in \mathcal{R}}\frac{\mu(T\cap R)}{\mu(T)} \langle f \rangle_R - \langle f \rangle_Q\right| + \left|\langle f \rangle_Q - \langle f \rangle_S\right| \\
& =: \mathrm{I} + \mathrm{II} + \mathrm{III}. \\
\end{align*}  
By a computation entirely analogous to the ones above, we can see that 
$$
\mathrm{III} \lesssim \|f\|_{\mathrm{RBMO}_{\Sigma^j}^*} \leq \delta(T,S) \; \|f\|_{\mathrm{RBMO}_{\Sigma^j}^*} .
$$
For $\mathrm{I}$, we use the doubling property of $T$ in the following way: 
\begin{align*}
\mathrm{I} & = \frac{1}{\mu(T)} \left| \int_{T} f \: d\mu - \sum_{R \in \mathcal{R}}\langle f \rangle_R \int_{T\cap R} d\mu \right| = \frac{1}{\mu(T)} \left|\sum_{R \in \mathcal{R}} \int_{T \cap R} (f- \langle f \rangle_R)\: d\mu \right|  \\
& \leq \frac{1}{\mu(T)} \sum_{R \in \mathcal{R}} \int_{T \cap R} \left|f- \langle f \rangle_R\right| d\mu \leq \frac{1}{\mu(T)} \sum_{R \in \mathcal{R}} \frac{\mu(R)}{\mu(R)} \int_{R} \left|f- \langle f \rangle_R\right| d\mu \\
& \leq \frac{\sum_{R \in \mathcal{R}} \mu(R)}{\mu(T)} \|f\|_{\mathrm{RBMO}_{\Sigma^j}^*} \leq \frac{\mu(10T)}{\mu(T)} \|f\|_{\mathrm{RBMO}_{\Sigma^j}^*} \lesssim \|f\|_{\mathrm{RBMO}_{\Sigma^j}^*} \leq \delta(T,S) \; \|f\|_{\mathrm{RBMO}_{\Sigma^j}^*}, \\
\end{align*} 
since all the cubes $R \in \mathcal{R}$ are contained in $10T$. Finally, we can readily check that 
$$
\mathrm{II} \leq \left|\sum_{R \in \mathcal{R}}\frac{\mu(T\cap R)}{\mu(T)} (\langle f \rangle_R - \langle f \rangle_Q) \right| \leq  \sup_{R \in \mathcal{R}}|\langle f \rangle_R - \langle f \rangle_Q|  \lesssim \left[ \sup_{R \in \mathcal{R}} \delta(R,Q) \right] \|f\|_{\mathrm{RBMO}_{\Sigma^j}^*} .
$$
Therefore, it is enough to check that for $R\in \mathcal{R}$ we have
\begin{equation}\label{eqKQR}
\delta(R,Q) \lesssim \delta(T,S).
\end{equation}
This can be checked via the following calculation:
\begin{align*}
 \delta(R,Q) & = 1 + \int_{\alpha B_Q \setminus \alpha B_R } \frac{1}{|y-x_{B_R}|^n}\; d\mu(y) \\
 & = 1+ \int_{\alpha B_Q \setminus (2S\cup  \alpha B_R) } \frac{1}{|y-x_{B_R}|^n}\; d\mu(y) + \int_{2S \setminus (2T \cup \alpha B_R) } \frac{1}{|y-x_{B_R}|^n}\; d\mu(y)  \\
 & \;\;\;\;\;\;\; + \int_{ 2T \setminus \alpha B_R } \frac{1}{|y-x_{B_R}|^n}\; d\mu(y) =: 1 + \mathrm{I}'+ \mathrm{II}'+ \mathrm{III}' .\\
\end{align*}
$\mathrm{I}'$ and $\mathrm{III}'$ are bounded above by an absolute constant since the pairs $B_Q$ and $S$ (on the one hand) and $T$ and $B_R$ (on the other) have comparable radii and side length, respectively. For term $\mathrm{II}'$ we have
$$ 
\mathrm{II}' \lesssim \int_{4S \setminus 2T } \frac{1}{|y-x_{T}|^n}\; d\mu(y) \sim \delta(T,S).
$$
Therefore, we have \eqref{eqKQR} and the proof is complete. \qed

\begin{remark}
Contrary to what happens in the Lebesgue measure case, theorem \ref{thmBMOdyad} does not imply that the norm in $\mathrm{RBMO}$ can be computed as the average of translates and dilates of $\mathrm{RBMO}_\Sigma$ for a given $\Sigma$.
\end{remark}

\begin{remark}
Theorem \ref{thmBMOdyad} yields the chain of inclusions
$$
L^\infty(\mu) \subsetneq \mathrm{RBMO} = \bigcap_{j=1}^N \mathrm{RBMO}_{\Sigma^j}^* \subsetneq \bigcap_{j=1}^N \mathrm{RBMO}_{\Sigma^j}.
$$
Since both $L^\infty(\mu)$ and $\mathrm{RBMO}_{\Sigma_j}$ |for all $j$| interpolate with the $L^p$ scale, all spaces in the display above interpolate as well.
\end{remark}


\appendix
\section{Usual dyadic cubes in the nondoubling setting} \label{secappendix}

The usual one third trick can be applied in the nondoubling setting so long as the cubes involved are are all doubling. This can be applied to the $\mathrm{RBMO}$ norm. We shall use the following version of the trick, that is also used in its more standard version in section \ref{sec1}:
 
\begin{lemma}\label{lemma2cubos}
There exist $3^d$ dyadic systems $\mathcal{D}^1, \mathcal{D}^2, \cdots, \mathcal{D}^{3^d}$ on $\mathbb{R}^d$ such that for all pairs of cubes $Q_1$ and $Q_2$ there exists $1 \leq k \leq 3^d$ and cubes $T_1, T_2 \in \mathcal{D}^k$ such that
$$
Q_1 \subset T_1 \subset 6Q_1 \; \mathrm{and} \: Q_2 \subset T_2 \subset 6Q_2.
$$ 
\end{lemma}

Lemma \ref{lemma2cubos} is essentially known. Its proof is a minimal variation of lemma 2.5 of \cite{hytonen-lacey-perez} (the only difference with the result there is the fact that the cubes $Q_1$ and $Q_2$ in the statement of lemma \ref{lemma2cubos} need not be dyadic).

\begin{remark}
Following the arguments in \cite{conde2013} (see also \cite{condealonsoPhd}) one can see that the optimal number of dyadic systems such that lemma \ref{lemma2cubos} holds is $2d+1$.
\end{remark}

If $\mathcal{D}$ is a dyadic filtration, we can define the dyadic version of $\mathrm{RBMO}$ by
\begin{align*}
\|f\|_{\mathrm{RBMO}_{\alpha,\beta,\mathcal{D}}} & := \sup_{\begin{subarray}{c} Q \in \mathcal{D}\\ Q \; (\alpha,\beta)-\mathrm{doubling} \end{subarray}} \frac{1}{\mu(Q)} \int_Q \left| f - \langle f \rangle_Q \right| d\mu \\
& + \sup_{\begin{subarray}{c} Q,R \in \mathcal{D}\\ Q,R \; (\alpha,\beta)-\mathrm{doubling} \end{subarray}} \left|\frac{\langle f \rangle_Q - \langle f \rangle_R}{\delta(Q,R)}\right|.
\end{align*}
Immediately, we get the following:

\begin{corollary} \label{usualddyadcubes}
Fix $\alpha\geq 2$ and $\beta > (6\alpha)^d$. Let $\mathcal{D}^j$, $1\leq j \leq 2d+1$, be the dyadic filtrations given by lemma \ref{lemma2cubos}. Then we have
$$
\|f\|_{\mathrm{RBMO}} \sim  \sup_{1\leq j \leq 3^d} \|f\|_{\mathrm{RBMO}_{\alpha,\beta,\mathcal{D}^j}}.
$$
\end{corollary}

\begin{proof}
First, it is immediate that
$$
\|f\|_{\mathrm{RBMO}_{\alpha,\beta,\mathcal{D}^j}} \leq \|f\|_{\mathrm{RBMO}}
$$
for all $j$. For the reverse inclusion, we may assume that the cubes in the definition of the $\mathrm{RBMO}$ norm are $(6\alpha,\beta)$-doubling. Then, given such a cube $Q$ we may use lemma \ref{lemma2cubos} to find an index $k$ and a cube $T \in \mathcal{D}^k$ such that $Q \subset T \subset 6Q$. We know that
$$
\mu(\alpha T) \leq \mu(6\alpha Q) \leq \beta \mu(Q) \leq \beta \mu(T),
$$
so $T$ is $(\alpha,\beta)$-doubling. Therefore, we can estimate
\begin{align*}
\frac{1}{\mu(Q)} \int_Q \left| f - \langle f \rangle_Q\right| d\mu & \leq \frac{1}{\mu(Q)} \int_Q \left| f - \langle f \rangle_T\right| d\mu + \left| \langle f \rangle_{T} - \langle f \rangle_Q \right| \leq 2\frac{1}{\mu(Q)} \int_Q \left| f - \langle f \rangle_T\right| d\mu \\
& \leq 2\frac{\mu(T)}{\mu(Q)} \frac{1}{\mu(T)} \int_T \left| f - \langle f \rangle_T\right| d\mu  \leq 2\frac{\mu(6\alpha Q)}{\mu(Q)} \frac{1}{\mu(T)} \int_T \left| f - \langle f \rangle_T\right| d\mu \\
& \lesssim \|f\|_{\mathrm{RBMO}_{\alpha,\beta,\mathcal{D}^k}}.
\end{align*}
Finally, given $(6\alpha,\beta)$-doubling cubes $Q$ and $R$ we apply lemma \ref{lemma2cubos} with $Q_1=Q$ and $Q_2=R$. Then, by lemma \ref{lemmaKQR} and the computation above we find that
\begin{align*}
\left| \langle f \rangle_Q - \langle f \rangle_R \right| & \leq \left| \langle f \rangle_Q - \langle f \rangle_{T_1} \right| + \left| \langle f \rangle_{T_2} - \langle f \rangle_R \right| + \left| \langle f \rangle_{T_1} - \langle f \rangle_{T_2} \right| \\
& \leq \frac{1}{\mu(Q)} \int_Q \left| f - \langle f \rangle_{T_1}\right| d\mu + \frac{1}{\mu(R)} \int_R \left| f - \langle f \rangle_{T_2}\right| d\mu +  \left| \langle f \rangle_{T_1} - \langle f \rangle_{T_2} \right| \\  
& \lesssim \|f\|_{\mathrm{RBMO}_{\alpha,\beta,\mathcal{D}^k}} + \delta(T_1,T_2) \|f\|_{\mathrm{RBMO}_{\alpha,\beta,\mathcal{D}^k}}  \\
& \lesssim \delta(Q,R) \|f\|_{\mathrm{RBMO}_{\alpha,\beta,\mathcal{D}^k}}.
\end{align*}
This shows that
$$
\|f\|_{\mathrm{RBMO}} \lesssim \sup_{1\leq j \leq 3^3} \|f\|_{\mathrm{RBMO}_{\alpha,\beta,\mathcal{D}^j}}
$$
\end{proof}
Corollary \ref{usualddyadcubes} is more similar to Mei's statement in \cite{mei2003} and simpler, because it only relies on usual dyadic cubes. However, the spaces $\mathrm{RBMO}_{\mathcal{D}^j}$ are not martingale $\mathrm{BMO}$ spaces in general (in fact, the quantity $\|f\|_{\mathrm{RBMO}_{\alpha,\beta,\mathcal{D}}}$ need not be a norm modulo constants). It is far from clear whether they interpolate with the $L^p$ scale or not in case $\mu$ is not doubling. This indicates that our approach via David-Mattila construction yields a more useful result.


 \bibliography{BibliographyRBMOdyad}{}
 \bibliographystyle{abbrv}

\end{document}